\documentclass[12pt]{amsart}
\usepackage{mathrsfs}

\usepackage{amssymb,amsfonts,amsthm,amsmath}
\usepackage{epsfig}
\usepackage[utf8]{inputenc}
\usepackage{mathrsfs}
\usepackage{lscape}
\usepackage{amssymb,amsmath,graphicx,color,textcomp, amsthm,bbm,bbold, enumerate,booktabs}
\usepackage{chngcntr}
\usepackage{apptools}
\AtAppendix{\counterwithin{theorem}{section}}
\definecolor{Red}{cmyk}{0,1,1,0}

\definecolor{verde}{cmyk}{1,0,1,0}

\definecolor{loka}{cmyk}{.5,0,1,.5}
\definecolor{azul}{cmyk}{1,1,0,0}

\numberwithin{equation}{section}

\newcommand{\be}{\begin{equation}}
\newcommand{\ee}{\end{equation}}

\newtheorem{definition}{Definition}

\newtheorem{theorem}{Theorem}
\newtheorem{remark}{Remark}
\newtheorem{lemma}{Lemma}

\begin{document}
\title{Fractional order pseudoparabolic partial differential equation: Ulam-Hyers stability}
\author{J. Vanterler da C. Sousa$^1$}
\address{$^1$ Department of Applied Mathematics, Institute of Mathematics,
 Statistics and Scientific Computation, University of Campinas --
UNICAMP, rua S\'ergio Buarque de Holanda 651,
13083--859, Campinas SP, Brazil\newline
e-mail: {\itshape \texttt{vanterlermatematico@hotmail.com, capelas@ime.unicamp.br }}}
\author{E. Capelas de Oliveira$^1$}

\begin{abstract} Using Gronwall inequality we will investigate the Ulam-Hyers and generalized Ulam-Hyers-Rassias stabilities for the solution of a fractional order pseudoparabolic partial differential equation.	

\end{abstract}
\maketitle

\section{Introduction} 
In 1695, Leibniz formulated a question, addressed to l'Hospital, involving a possible generalization of the derivative of whole order to a derivative of order, in principle, arbitrary, and may even be complex. L'Hospital returned the question to Leibniz, questioning him in the case where the order of the derivative was middle and what a possible interpretation might be. In an audacious and prophetic response, Leibniz presents the result and states this is apparently a paradox that one day will generate several important consequences. Thus, the fractional calculus (FC) begins, and since then numerous derivative and integral formulations have been introduced 
\cite{almeida,HER,katu,KSTJ,IP,SAMKO,ZE1,ZE2}. Among the various formulations, recently Sousa and Oliveira \cite{ZE1}, introduced the so-called $\psi$-Hilfer fractional derivative of a single variable that generalizes a wide class of other formulations of fractional derivatives such as: Riemann-Liouville, Caputo, Hilfer, Riesz and other more recent, for example, generalized Caputo derivative \cite{oliveira2017}. In order to study the stability 
of solutions of partial differential equations by means of a fractional derivative, especially the $\psi$-Hilfer type, there is a need to extend the definition to $N$ variables \cite{ZE3}.

We mention a history, similar to FC, how the first idea of stability of functional equations came about. In 1940,
Ulam and Hyers exchanged correspondences on the stability study of solutions of differential equations. Since then, this theme, has been a motivator for many researchers, especially mathematicians \cite{hyers,ulam}. Subsequently, this type of stability has come to be called Ulam-Hyers stability. Study the various types of stability, be they of the type Ulam-Hyers, Ulam-Hyers-Rassias, semi-Ulam-Hyers-Rassias, Ulam-Hyers-Mittag-Leffler, $\delta$-Ulam-Hyers-Rassias \cite{estp6,estp2,est2,esto1,estp4,estp7,estp1,estp8,estp5,esto4,esto5,esto6,esto7,esto8} of solutions of 
partial and/or ordinary differential equations, by means of fractional derivatives, has been growing because of the huge quantity of papers, published during this period, justifying the importance of this particular area of mathematical analysis, in particular, regarding the FC.

Several researchers started to focus on the study of Ulam-Hyers and generalized Ulam-Hyers-Rassias stabilities, among others. However, we have mentioned that Abbas and Benchohra \cite{abbas8,abbas6,abbas7} are researchers whose study of stabilities is directed to solutions of fractional partial differential equations. Among some 
works that these authors have performed, we highlight: the study involving partial differential equations of the hyperbolic type and partial differential equations with delay in time, \cite{abbas8,abbas,abbas1,abbas5,abbas6,abbas4,abbas2,abbas7,abbas3}. In this sense, numerous studies were carried out and FC becomes a new area of application and consequently gains more space within the mathematical analysis. Other researchers such as Long et al. \cite{long}, Ahmad et al. \cite{ahmad}, Choung et al. \cite{chuong} and Zhang \cite{yan} have published works related to the stability of solutions of partial differential equations, in some cases involving fractional neutral stochastic partial integro-differential equations.

In this paper we will consider the following fractional order pseudoparabolic partial differential equation
\begin{equation}
\dfrac{\partial _{\beta ;\psi }^{3\alpha }u}{\partial _{\beta ;\psi }^{\alpha }x^{2\alpha }\partial _{\beta ;\psi }^{\alpha }y^{\alpha }}\left( x,y\right) 
=f\left( x,y,u\left( x,y\right) ,\dfrac{\partial _{\beta ;\psi }^{\alpha }u}{\partial _{\beta ;\psi }^{\alpha }y^{\alpha }}\left( x,y\right) ,
\dfrac{\partial _{\beta ;\psi }^{2\alpha }u}{\partial _{\beta
;\psi }^{\alpha }x^{2\alpha }}\left( x,y\right) \right) ,  \label{eq3}
\end{equation}%
where $\dfrac{\partial _{\beta ;\psi }^{3\alpha }u}{\partial _{\beta ;\psi }x^{2\alpha }\partial _{\beta ;\psi }y^{\alpha }}\left( \cdot ,\cdot \right)$ is the $\psi $-Hilfer fractional partial derivative \cite{ZE3} with the parameters $\dfrac{2}{3}<\alpha \leq 1$, $0\leq \beta \leq 1$ and $\text{ }0\leq x<a,\text{ }0\leq y<b$, being $f\in C\left( \left[ 0,a\right) \times \left[ 0,b\right) \times \mathbb{B}^{3},\mathbb{B}\right) $ and $(\mathbb{B},\left\vert \cdot \right\vert )$ a real or complex Banach space.

The main motivation of this paper is to present a study on the Ulam-Hyers and generalized Ulam-Hyers-Rassias stabilities of the solution of a fractional order pseudoparabolic partial differential equation. For this purpose, we use the $\psi$-Hilfer fractional derivative of $N$ variables and the Gronwall inequality, in order to contribute to the study of stabilities and provide a new and interesting result for future research.

The paper is organized as follows: In section 2, we will present the definition of the $\psi$-Riemann-Liouville fractional integral of a function relative to another function of $N$ variables and the $\psi$-Hilfer fractional partial derivative. Moreover, through the $\psi$-Hilfer fractional partial derivative, we will present a new version for the definition of Ulam-Hyers and Ulam-Hyers-Rassias stabilities, the Gronwall inequality and some remarks. In section 3, our main result, we will study the Ulam-Hyers and generalized Ulam-Hyers-Rassias stabilities of the solution of a fractional order pseudoparabolic partial differential equation. Concluding remarks close the paper.

\section{Preliminaries} 
In this section we will present the definition of $\psi$-Riemann-Liouville fractional integral and $\psi$-Hilfer fractional derivative of $N$ variables, as well as Gronwall's lemma, fundamental in the study of solutions of differential equations. In this sense, the definitions of Ulam-Hyers and generalized Ulam-Hyers-Rassias stabilities are introduced in an adapted version associated with the type of partial differential equation to be studied.

First, we present the definition of $\psi$-Riemann-Liouville fractional integral, fundamental to the $\psi$-Hilfer fractional derivative approach.

\begin{definition}{\rm \cite{ZE3}} Let $\theta =\left( \theta _{1},\theta _{2},...,\theta _{N}\right) $ and $ \alpha =\left( \alpha _{1},\alpha _{2},..., \alpha _{N}\right) $,  where $ 0<\alpha _{1},\alpha _{2},...,\alpha _{N}<1$, $N\in \mathbb{N}$. Also put $\widetilde{I}=I_{1}\times I_{2}\times  \cdot \cdot \cdot \times I_{N}=\left[\theta _{1},a_{1}\right] \times \left[ \theta _{2},a_{2}\right] \times \cdot\cdot \cdot \times \left[ \theta _{N},a_{N}\right] ,$ where $a_{1},a_{2},...,a_{N}$ and $\theta _{1},\theta _{2},...,\theta _{N}$are positive constants. Also let $\psi \left( \cdot \right) $ be an increasing and positive monotone function on $\left( \theta _{1},a_{1}\right], \left(\theta _{2},a_{2}\right] ,...,\left( \theta _{N},a_{N}\right]$, having a continuous 
derivative $\psi ^{\prime }\left( \cdot \right) $ on  $\left( \theta _{1},a_{1}\right], \left( \theta _{2},a_{2}\right] ,...,\left( \theta  _{N},a_{N}\right]$. The $\psi$-Riemann-Liouville partial integral of $N$ variables $u=\left( u_{1},u_{2},...,u_{N}\right) \in L^{1}\left( \widetilde{I }\right)$  is defined by  
\begin{equation}
I_{\theta ,x}^{\alpha ;\psi }u\left(x\right) =\frac{1}{\Gamma \left({\alpha }_{j}\right) }\int\int\cdot \cdot \cdot 
\int_{\widetilde{I}}\psi ^{\prime }\left({s}_{j}\right) \left( \psi \left( {x}_{j}\right) -\psi \left( {s}_{j}\right) \right) ^{{\alpha }_{j}-1}u\left({s}_{j}\right) d{s}_{j}, \label{eq30}
\end{equation}
with $\psi ^{\prime }\left({s}_{j}\right) \left(\psi \left({x}_{j}\right) -\psi \left({s}_{j}\right) 
\right) ^{{\alpha }_{j}-1}=\psi ^{\prime }\left( s_{1}\right) \left( \psi \left( x_{1}\right) -\psi \left( s_{1}\right) \right) ^{\alpha_{1}-1}\psi ^{\prime }\left( s_{2}\right) \left( \psi \left( x_{2}\right)-\psi \left(s_{2}\right) \right) ^{\alpha _{2}-1}\cdot \cdot \cdot \psi ^{\prime }\left(s_{N}\right) 
\left( \psi \left( x_{N}\right) -\psi \left(s_{N}\right)\right) ^{\alpha _{N}-1}$,  and using the notation $\Gamma \left( {\alpha }_{j}\right) =\Gamma \left( \alpha_{1}\right) \Gamma \left( \alpha_{2}\right) \cdot \cdot \cdot \Gamma \left(\alpha _{N}\right)$, $u\left({s}_{j}\right)=u\left( s_{1}\right) u\left( s_{2}\right) \cdot\cdot \cdot u\left(s_{N}\right) $ and $d{s}_{j}=ds_{1}ds_{2}\cdot\cdot \cdot ds_{N}$, $j\in \left\{ 1,2,...,N\right\} $ with $N\in \mathbb{N}$.
\end{definition}

From the fractional patial integral Eq.(\ref{eq30}), it is possible to obtain other fractional partial integrals, that is Erd\`elyi-Kober fractional partial integral, Katugampola fractional partial integral, Weyl fractional partial integral, among others. In addition, each fractional partial integral obtained here is an extension of its respective fractional integral \rm {\cite{HER,KSTJ,SAMKO,ZE1}.

In particular, taking $N=2$ and $\theta_{1}=\theta_{2}=0$ in \textrm{{Eq.(\ref{eq30})}} we have the fractional partial integral that will be used in what follows, 
\begin{eqnarray}
I_{\theta }^{\alpha ;\psi }u\left( x_{1},x_{2}\right) &=&\frac{1}{\Gamma \left( \alpha _{1}\right) \Gamma \left( \alpha _{2}\right) }\int_{0}^{x_{1}}\int_{0}^{x_{2}} 
\psi ^{\prime }\left( s_{1}\right)
\psi^{\prime }\left( s_{2}\right) \left( \psi \left( x_{1}\right) -\psi\left(s_{1}\right) \right) ^{\alpha _{1}-1}  \notag  \label{eq31} \\&&\left( \psi \left( x_{2}\right) 
-\psi \left( s_{2}\right) \right)^{\alpha_{2}-1}u\left( s_{1},s_{2}\right) ds_{1}ds_{2},
\end{eqnarray}
with $0<\alpha _{1},\alpha _{2}\leq 1$.

Also, we have 
\begin{equation}  \label{eq32}
I_{0+,x_{1}}^{\alpha _{1};\psi }u\left( x_{1},x_{2}\right) =\frac{1}{\Gamma \left( \alpha _{1}\right) }\int_{0}^{x_{1}}\psi ^{\prime }\left( s_{1}\right) 
\left( \psi \left( x_{1}\right) -\psi \left( s_{1}\right) \right) ^{\alpha _{1}-1}u\left( s_{1},s_{2}\right) ds_{1}
\end{equation}
and 
\begin{equation}  \label{eq33}
I_{0+,x_{2}}^{\alpha _{2};\psi }u\left( x_{1},x_{2}\right) =\frac{1}{\Gamma \left( \alpha _{2}\right) }\int_{0}^{x_{2}}\psi ^{\prime }\left( s_{2}\right) 
\left( \psi \left( x_{2}\right) -\psi \left( s_{2}\right) \right) ^{\alpha _{2}-1}u\left( s_{1},s_{2}\right) ds_{2},
\end{equation}
with $0<\alpha _{1},\alpha _{2}\leq 1$.

Using the $\psi$-Riemann-Liouville fractional partial integral, we present the $\psi$-Hilfer fractional partial derivative. The following definition is an extension of the recent fractional derivative of a variable recently introduced by Sousa and Oliveira \cite{ZE1}.

\begin{definition} {\rm \cite{ZE3}} Let $\theta =\left( \theta_{1},\theta_{2},...,\theta_{N}\right) $ and $\alpha =\left( \alpha _{1},\alpha _{2},...,\alpha _{N}\right) $, 
where $0<\alpha_{1},\alpha _{2},...,\alpha _{N}<1$, $N\in \mathbb{N}$. Also put $\widetilde{I}=I_{a_{1}}\times I_{a_{2}}\times \cdot \cdot \cdot 
\times I_{a_{N}}=\left[ \theta_{1},a_{1}\right] \times \left[ \theta_{2},a_{2} \right] \times \cdot \cdot \cdot \times \left[ \theta_{N},a_{N}\right]$,
where $a_{1},a_{2},...,a_{N}$ and $\theta_{1}, \theta_{2},..., \theta_{N}$
are positive constants. Also let $u,\psi \in C^{n}\left( \widetilde{I},  \mathbb{R} \right) $ two functions such that $\psi $ is increasing and $\psi ^{\prime } \left( x_{i}\right) \neq 0,$ $ i\in \left\{ 1,2,...,N\right\}$, $x_{i}\in \widetilde{I},$ $N\in \mathbb{N}$. The $\psi$-Hilfer fractional partial derivative of $N$ variables denoted by $^{H}\mathbb{D}_{\theta, x }^{\alpha ,\beta ;\psi }\left( \cdot \right) $ of a function, of order $\alpha $ and type $0\leq \beta _{1},\beta_{2},...,\beta _{N}\leq 1,$ is defined by  
\begin{equation}  \label{eq34}
^{H}\mathbb{D}_{\theta, {x}}^{\alpha ,\beta ;\psi }u\left({x}\right) =I_{\theta, {x}_{j}}^{\beta \left(1-\alpha \right) ;\psi } \left( \frac{1}{\psi ^{\prime }\left({x} _{j}\right) }\frac{\partial ^{N}}{\partial {x}_{j}}\right) I_{\theta ,{x}_{j}}^{\left( 1-\beta \right) \left(1-\alpha \right) ; \psi }u \left({x}_{j}\right),
\end{equation}
with $\partial{x}_{j}=\partial x_{1}\partial x_{2}\cdot \cdot \cdot \partial x_{N}$ and $\psi ^{\prime }\left({x}_{j}\right) =\psi ^{\prime }\left( x_{1}\right) \psi ^{\prime }\left( x_{2}\right) \cdot \cdot\cdot \psi ^{\prime }\left( x_{N}\right) ,$ $j\in \left\{
1,2,...,N\right\}$, $N\in \mathbb{N}$.
\end{definition}

In the same way that a large class of fractional partial integrals can be obtained, as particular cases, it is also possible for the $\psi$-Hilfer fractional partial derivative. This vast class will be omitted here, however we suggest the following papers \cite{ZE1,ZE3}.

Taking $N=2$ in \textrm{{Eq.(\ref{eq34})}}, we present the partial fractional derivative that will be used in this paper, 
\begin{equation}  \label{eq35}
^{H}\mathbb{D}_{\theta }^{\alpha ,\beta ;\psi }u\left( x_{1},x_{2}\right) =I_{\theta }^{\beta \left( 1-\alpha \right) ;\psi }\left( \frac{1}{\psi ^{\prime }
\left( x_{1}\right) \psi ^{\prime }\left( x_{2}\right) }\frac{ \partial ^{2} }{\partial x_{1}\partial x_{2}}\right) I_{\theta }^{\left(
1-\beta \right) \left( 1-\alpha \right) ;\psi }u\left( x_{1},x_{2}\right).
\end{equation}

Also, we use the following notation 
\begin{equation}  \label{eq36}
^{H}\mathbb{D}_{\theta }^{\alpha ,\beta ;\psi }u\left( x_{1},x_{2}\right) =\frac{\partial _{\beta ;\psi }^{3\alpha }u}{\partial _{\beta ;\psi}x^{2\alpha }
\partial_{\beta ;\psi }y^{\alpha }}\left( x_{1},x_{2}\right).
\end{equation}

We remark that, the following Gronwall lemma is an important tool in proving the main results of this paper.

\begin{lemma}\label{l1}{\rm\cite{gronwall}} One assumes that
\begin{enumerate}
\item  $u,v,h\in C\left( [a,b],\mathbb{R}_{+}\right)$;
\item For any $t\geq a$ and $\psi \left( t\right) $ is increasing and $\psi ^{\prime }\left( t\right) $ for all $t\in [a,b]$ one has
\begin{equation*}
u\left( t\right) \leq v\left( t\right) +h\left( t\right) \int_{a}^{t}\psi ^{\prime }\left( s\right) \left( \psi \left( t\right) -\psi \left( s\right) 
\right) ^{\alpha -1}u\left( s\right) ds,
\end{equation*}
\item $h\left( t\right) $ is nonnegative and nondecreasing.
\end{enumerate}

Then, we have
\begin{equation*}
u\left( t\right) \leq v\left( t\right) \mathbb{E}_{\alpha }\left[ h\left( t\right) \Gamma \left( \alpha \right) \left( \psi \left( t\right) -\psi 
\left( a\right) \right) ^{\alpha }\right],
\end{equation*}
$\text{ for any }t\geq a$ and being $\mathbb{E}_{\alpha}(\cdot)$ the one-parameter Mittag-Leffler function.
\end{lemma}

To facilitate the development of the calculations, we introduce the following notation:
\begin{enumerate}
\item $\mathcal{W}_{v_{1},v_{2}}^{p,v}f\left( x,y\right) :=\displaystyle\int_{0}^{p}f\left( x,y,v\left(
x,y\right) ,v_{1}\left( x,y\right) ,v_{2}\left( x,y\right) \right) dp$;

\item $\mathcal{W}_{u_{1},u_{2}}^{p,u}f\left( x,y\right) :=\displaystyle\int_{0}^{p}f\left( x,y,u\left(
x,y\right) ,u_{1}\left( x,y\right) ,u_{2}\left( x,y\right) \right) dp$;

\item $\widetilde{\mathcal{W}}_{v,u}^{p}f\left( x,y\right) :=\displaystyle\int_{0}^{p}\left\vert 
\begin{array}{c}
f\left( x,y,v\left( x,y\right) ,v_{1}\left( x,y\right) ,v_{2}\left( x,y\right)\right) \\-f\left( x,y,u\left( x,y\right) ,u_{1}\left( x,y\right) ,u_{2}
\left(x,y\right) \right) 
\end{array}
\right\vert dp$

\item $\mathcal{W}^{p}\varphi \left( x,y\right) :=\displaystyle\int_{0}^{p}\varphi \left( x,y\right) dp$;

\item $\Psi ^{\gamma }\left( x,0\right) :=\dfrac{\left( \psi \left( x\right) -\psi\left( 0\right) \right) ^{\gamma -1}}{\Gamma \left( \gamma \right) }$;

\item $\Psi ^{\gamma }\left( 0,y\right) :=\dfrac{\left( \psi \left( y\right) -\psi \left( 0\right) \right) ^{\gamma -1}}{\Gamma \left( \gamma \right) }$;

\end{enumerate}

All of the above items exist and are well defined.

Let $a,b\in \left( 0,\infty \right] ,$ $\varepsilon >0,$ $\varphi \in C\left( \left[ 0,a\right) \times \left[ 0,b\right), \mathbb{R}_{+}\right)$ and $(\mathbb{B}, 
\left\vert \cdot\right\vert)$ be a real or complex Banach
space.

We consider the following inequalities
\begin{equation}
\left\vert \dfrac{\partial _{\beta ;\psi }^{3\alpha }v}{\partial _{\beta ;\psi }^{\alpha }x^{2\alpha }\partial _{\beta ;\psi }^{\alpha }y^{\alpha }}
\left( x,y\right) -f\left( x,y,v\left( x,y\right) ,\frac{\partial _{\beta ;\psi }^{\alpha }v}{\partial _{\beta ;\psi }^{\alpha }y^{\alpha }}\left(
x,y\right) ,\frac{\partial _{\beta ;\psi }^{2\alpha }v}{\partial _{\beta ;\psi }^{\alpha }x^{2\alpha }}\left( x,y\right) \right) \right\vert \leq \varepsilon ,  \label{eq4}
\end{equation}%
$x\in \left[ 0,a\right) ,$ $y\in \left[ 0,b\right) $; 
\begin{equation}
\left\vert \frac{\partial _{\beta ;\psi }^{3\alpha }v}{\partial _{\beta ;\psi }^{\alpha }x^{2\alpha }\partial _{\beta ;\psi }^{\alpha }y^{\alpha }}
\left( x,y\right) -f\left( x,y,v\left( x,y\right) ,\frac{\partial _{\beta;\psi }^{\alpha }v}{\partial _{\beta ;\psi }^{\alpha }y^{\alpha }}\left(
x,y\right) ,\frac{\partial _{\beta ;\psi }^{2\alpha }v}{\partial _{\beta;\psi }^{\alpha }x^{2\alpha }}\left( x,y\right) \right) \right\vert \leq\varphi 
\left( x,y\right) ,  \label{eq5}
\end{equation}%
$x\in \left[ 0,a\right) ,$ $y\in \left[ 0,b\right)$, with $\dfrac{2}{3}<\alpha \leq 1$ and $0\leq \beta \leq 1$.

Note that, a function $u:\left[ 0,a\right) \times \left[ 0,b\right) \rightarrow \mathbb{B}$ is a solution of \textrm{{Eq.(\ref{eq3})}} if 
$u\in C\left( \left[ 0,a\right) \times \left[ 0,b\right) \right) \cap C^{1}\left( \left[ 0,a\right) \times \left[ 0,b\right) \right) ,$ 
$\dfrac{\partial_{\beta ;\psi }^{2\alpha }}{\partial _{\beta ;\psi }^{\alpha }x^{2\alpha }} \in C\left( \left[ 0,a\right) \times \left[ 0,b\right) 
\right) $, $\dfrac{\partial _{\beta ;\psi }^{3\alpha }}{\partial _{\beta ;\psi }^{\alpha
}x^{2\alpha }\partial _{\beta ;\psi }^{\alpha }y^{\alpha }}\in C\left( \left[ 0,a\right) \times \left[ 0,b\right) \right) $ and $u$ satisfies the \textrm{{Eq.(\ref{eq3})}}.

The following definitions were adapted for the $\psi$-Hilfer fractional derivative for two variables and we used in \cite{lungu,ZE3}.

\begin{definition}
The solution of {\rm{{Eq.(\ref{eq3})}}} is Ulam-Hyers stability if there exist real numbers $C_{f}^{1},C_{f}^{2}$ and $C_{f}^{3}>0$ such that for any $ \varepsilon >0$ 
and for any solution $v$ to the inequality \rm{{Eq.(\ref{eq4})}} with 
\begin{equation*}
\left\vert v\left( x,y\right) -u\left( x,y\right) \right\vert \leq C_{f}^{1}\varepsilon ,\text{ }
\end{equation*}%
\begin{equation*} 
\left\vert \dfrac{\partial _{\beta ;\psi }^{\alpha }v}{\partial _{\beta ;\psi }^{\alpha }y^{\alpha }}\left( x,y\right) -
\dfrac{\partial _{\beta ;\psi }^{\alpha }u}{\partial _{\beta ;\psi }^{\alpha }y^{\alpha }}\left(x,y\right) \right\vert \leq C_{f}^{2}\varepsilon ,
\end{equation*}
\begin{equation*}
\left\vert \dfrac{\partial _{\beta ;\psi }^{2\alpha }v}{\partial _{\beta ;\psi }^{\alpha }x^{2\alpha }}\left( x,y\right) -
\dfrac{\partial _{\beta ;\psi }^{2\alpha }u}{\partial _{\beta ;\psi }^{\alpha }x^{2\alpha }}\left( x,y\right) \right\vert \leq C_{f}^{3}\varepsilon,
\end{equation*}
$x\in \left[ 0,a\right) ,$ $y\in \left[ 0,b\right) $ with $\dfrac{2}{3}<\alpha \leq 1$ and $0\leq \beta \leq 1$.
\end{definition}

\begin{definition}
The solution of {\rm{{Eq.(\ref{eq3})}}} admits generalized Ulam-Hyers-Rassias stability if there exist real numbers 
$C_{f,\varphi }^{1},C_{f,\varphi }^{2}$ and $C_{f,\varphi }^{3}>0$ such that for any $
\varepsilon >0$ and for any solution $v$ to the inequality {\rm{{Eq.(\ref{eq5})}}} with 
\begin{equation*}
\left\vert v\left( x,y\right) -u\left( x,y\right) \right\vert \leq C_{f,\varphi }^{1}\varphi \left( x,y\right) ,\text{ }
\end{equation*}
\begin{equation*}
\left\vert \frac{\partial _{\beta ;\psi }^{\alpha }v}{\partial _{\beta ;\psi }^{\alpha }y^{\alpha }}\left( x,y\right) -
\frac{\partial _{\beta ;\psi }^{\alpha }u}{\partial _{\beta ;\psi }^{\alpha }y^{\alpha }}\left( x,y\right) \right\vert \leq C_{f,\varphi }^{2}\varphi \left( x,y\right) ,
\end{equation*}

\begin{equation*}
\left\vert \frac{\partial _{\beta ;\psi }^{2\alpha }v}{\partial _{\beta ;\psi }^{\alpha }x^{2\alpha }}\left( x,y\right) -
\frac{\partial _{\beta ;\psi }^{2\alpha }u}{\partial _{\beta ;\psi }^{\alpha }x^{2\alpha }}\left( x,y\right) \right\vert \leq C_{f,\varphi }^{3}\varphi \left( x,y\right) ,
\end{equation*}
$x\in \left[ 0,a\right) ,$ $y\in \left[ 0,b\right) $ with $\dfrac{2}{3}<\alpha \leq 1$ and $0\leq \beta \leq 1$.
\end{definition}

\begin{remark}\label{re1} A function $v$ is a solution to the inequality {\rm{{Eq.(\ref{eq4})}}} if, and only if, there exists a 
function $g\in C\left( \left[0,a\right)\times \left[ 0,b\right), \mathbb{B}\right)$, which depends on $v$, such that 

\begin{enumerate}
\item For all $\varepsilon >0$, $\left\vert g\left( x,y\right) \right\vert \leq \varepsilon $, $\forall x\in \left[ 0,a\right) $, $\forall y\in \left[ 0,b\right) $;

\item $\forall x\in \left[ 0,a\right) $, $\forall y\in \left[ 0,b\right) $, 
\begin{equation*}
\frac{\partial _{\beta ;\psi }^{3\alpha }v}{\partial _{\beta ;\psi }^{\alpha }x^{2\alpha }\partial _{\beta ;\psi }^{\alpha }y^{\alpha }}
\left( x,y\right) =f\left( x,y,v\left( x,y\right) ,\frac{\partial _{\beta ;\psi }^{\alpha }v}{\partial _{\beta ;\psi }^{\alpha }y^{\alpha }}\left( x,y\right) ,\frac{%
\partial _{\beta ;\psi }^{2\alpha }v}{\partial _{\beta ;\psi }^{\alpha}x^{2\alpha }}\left( x,y\right) \right) +g\left( x,y\right) 
\end{equation*}
with $\dfrac{2}{3}<\alpha \leq 1$ and $0\leq \beta \leq 1$. 
\end{enumerate}
\end{remark}

\begin{remark}\label{re2} A function $v$ is a solution to the inequality {\rm{{Eq.(\ref{eq5})}}} if, and only if, there exists a function 
$g\in C\left( \left[ 0,a\right) \times \left[ 0,b\right), \mathbb{B}\right)$, which depends on $v$, such that 

\begin{enumerate}
\item $\left\vert g\left( x,y\right) \right\vert \leq \varphi \left(x,y\right)$, $\forall x\in \left[ 0,a\right) $, $\forall y\in \left[ 0,b\right) $;

\item $\forall x\in \left[ 0,a\right) $, $\forall y\in \left[ 0,b\right) $, 
\begin{equation*}
\frac{\partial _{\beta ;\psi }^{3\alpha }v}{\partial _{\beta ;\psi }^{\alpha }x^{2\alpha }\partial _{\beta ;\psi }^{\alpha }y^{\alpha }}\left( x,y\right) =f
\left( x,y,v\left( x,y\right) ,\frac{\partial _{\beta ;\psi }^{\alpha }v}{\partial _{\beta ;\psi }^{\alpha }y^{\alpha }}\left( x,y\right) ,\frac{%
\partial _{\beta ;\psi }^{2\alpha }v}{\partial _{\beta ;\psi }^{\alpha}x^{2\alpha }}\left( x,y\right) \right) +g\left( x,y\right) 
\end{equation*}
with $\dfrac{2}{3}<\alpha \leq 1$ and $0\leq \beta \leq 1$ .
\end{enumerate}
\end{remark}

\begin{remark}\label{re3} If $v$ is a solution to the inequality {\rm{{Eq.(\ref{eq4})}}}, then $\left( v,v_{1},v_{2}\right) $, is a solution of the 
following system of integral inequalities:
\begin{eqnarray*}
&&\left\vert 
\begin{array}{c}
v\left( x,y\right) -\Psi^{\gamma }\left( 0,y\right)v\left( x,0\right) - \Psi^{\gamma }\left( x,0\right) v\left( 0,y\right) -\Psi^{\gamma }\left( x,0\right) 
xv_{x}\left( 0,y\right)\\ - I_{\theta}^{\alpha ;\psi }\displaystyle\left( \mathcal{W}_{v_{1},v_{2}}^{p,v}f\left( x,y\right)\right) 
\end{array}%
\right\vert  \\
&\leq &\varepsilon x\dfrac{\left( \psi \left( x\right) -\psi \left( 0\right) \right) ^{\alpha _{1}}\left( \psi \left( y\right) -\psi \left( 0\right) 
\right) ^{\alpha _{2}}}{\Gamma \left( \alpha _{1}+1\right) \Gamma \left(\alpha _{2}+1\right) };
\end{eqnarray*}%
\begin{eqnarray*}
&&\left\vert 
\begin{array}{c}
v_{1}\left( x,y\right) -\Psi^{\gamma }\left( x,0\right) v_{1}\left( 0,y\right) - \Psi^{\gamma }\left( x,0\right)xv_{1x}\left( 0,y\right) -
I_{0+,x}^{\alpha _{1};\psi }\displaystyle\left( \mathcal{W}_{v_{1},v_{2}}^{p,v}f\left( x,y\right)\right) 
\end{array}%
\right\vert  \\
&\leq &\varepsilon x\frac{\left( \psi \left( x\right) -\psi \left( 0\right)
\right) ^{\alpha _{1}}}{\Gamma \left( \alpha _{1}+1\right) };
\end{eqnarray*}%
\begin{equation*}
\left\vert 
\begin{array}{c}
v_{2}\left( x,y\right) -\Psi^{\gamma }\left( x,0\right) v_{2}\left(0,y\right) -I_{0+,y}^{\alpha _{2};\psi }\displaystyle
\left( \mathcal{W}_{v_{1},v_{2}}^{p,v}f\left( x,y\right)\right) 
\end{array}%
\right\vert \leq \varepsilon \frac{\left( \psi \left( y\right) -\psi \left(
0\right) \right) ^{\alpha _{2}}}{\Gamma \left( \alpha _{2}+1\right) };
\end{equation*}%
$ x\in \left[ 0,a\right) ,$ $y\in \left[ 0,b\right) $, $%
\dfrac{2}{3}<\alpha \leq 1$, $0\leq \gamma \leq 1$, $v_{1}=\dfrac{\partial
_{\beta ;\psi }^{\alpha }v}{\partial _{\beta ,;\psi }y^{\alpha }}$ and $%
v_{2}=\dfrac{\partial _{\beta ;\psi }^{2\alpha }v}{\partial _{\beta ,;\psi
}x^{2\alpha }}$.
\end{remark}

\begin{proof} From \rm{{Eq.(\ref{eq5})}} we have,
\begin{equation*}
\begin{array}{c}
v\left( x,y\right) -\Psi^{\gamma }\left( 0,y\right) v\left( x,0\right) + \Psi^{\gamma }\left( x,0\right) v\left( 0,y\right) +\Psi^{\gamma }
\left( x,0\right) xv_{x}\left( 0,y\right)-I_{\theta }^{\alpha ;\psi }\displaystyle\left(\mathcal{W}_{v_{1},v_{2}}^{p,v}f\left( x,y\right)
\right)\\+I_{\theta }^{\alpha ;\psi }\displaystyle\left( \mathcal{W}^{p}\varphi \left( x,y\right)\right) .
\end{array}
\end{equation*}

Then, we can write
\begin{eqnarray*}
&&\left\vert 
\begin{array}{c}
v\left( x,y\right) -\Psi^{\gamma }\left( 0,y\right)v\left( x,0\right) - \Psi^{\gamma }\left( x,0\right)v\left( 0,y\right) -\Psi^{\gamma }
\left( x,0\right)xv_{x}\left( 0,y\right)\\ -I_{\theta }^{\alpha ;\psi }\displaystyle\left( \mathcal{W}_{v_{1},v_{2}}^{p,v}f\left( x,y\right)\right) 
\end{array}%
\right\vert   \notag \\
&\leq &\left\vert I_{\theta }^{\alpha ;\psi }\displaystyle\left(\mathcal{W}^{p}\varphi \left( x,y\right)\right) \right\vert \leq I_{\theta }^{\alpha ;\psi }
\displaystyle\left( \int_{0}^{p}\varepsilon dp\right)  =\varepsilon x\dfrac{\left( \psi \left( x\right) -\psi \left( 0\right)
\right) ^{\alpha _{1}}\left( \psi \left( y\right) -\psi \left( 0\right)
\right) ^{\alpha _{2}}}{\Gamma \left( \alpha _{1}+1\right) \Gamma \left(
\alpha _{2}+1\right) }.
\end{eqnarray*}

In this sense, we have the inequalities 
\begin{eqnarray*}
&&\left\vert 
\begin{array}{c}
v_{1}\left( x,y\right) -\Psi^{\gamma }\left( x,0\right) v_{1}\left( 0,y\right) -\Psi^{\gamma }\left( x,0\right) xv_{x}\left( 0,y\right) -I_{0+;x}^{\alpha_{1} ;\psi }
\displaystyle\left( \mathcal{W}_{v_{1},v_{2}}^{p,v}f\left( x,y\right)\right) 
\end{array}%
\right\vert   \notag \\
&\leq &\varepsilon x\dfrac{\left( \psi \left( x\right) -\psi \left( 0\right) \right) ^{\alpha _{1}}}{\Gamma \left( \alpha _{1}+1\right) }
\end{eqnarray*}
and
\begin{eqnarray*}
&&\left\vert v_{2}\left( x,y\right) -\Psi^{\gamma }\left( 0,y\right) v_{2}\left( x,0\right) -I_{0+;y}^{\alpha_{2} ;\psi }f\left(x,y,v(x,y),v_{1}(x,y),v_{2}(x,y) 
\right) \right\vert   \notag \\
&\leq &\varepsilon \dfrac{\left( \psi \left( y\right) -\psi \left( 0\right) \right) ^{\alpha _{2}}}{\Gamma \left( \alpha _{2}+1\right) }.
\end{eqnarray*}
\end{proof}

\begin{remark}\label{re4} If $v$ is a solution of the {\rm{{Eq.(\ref{eq5})}}}, then $\left( v,v_{1},v_{2}\right)$ is a solution of the following system of 
integral inequalities: 
\begin{eqnarray*}
&&\left\vert 
\begin{array}{c}
v\left( x,y\right) -\Psi^{\gamma }\left( 0,y\right) v\left( x,0\right) - \Psi^{\gamma }\left( x,0\right)v\left( 0,y\right) - \Psi^{\gamma }\left( x,0\right) 
xv_{x}\left( 0,y\right) -I_{\theta }^{\alpha;\psi }\displaystyle\left(\mathcal{W}_{v_{1},v_{2}}^{p,v}f\left( x,y\right)\right) 
\end{array}
\right\vert   \notag \\
&\leq &I_{\theta }^{\alpha ;\psi }\left(\mathcal{W}^{p}\varphi \left( x,y\right)\right) ;
\end{eqnarray*}
\begin{eqnarray*}
&&\left\vert 
\begin{array}{c}
v_{1}\left( x,y\right) -\Psi^{\gamma }\left( x,0\right) v_{1}\left( 0,y\right) - \Psi^{\gamma }\left( x,0\right) xv_{1x}\left( 0,y\right)-I_{0+;x}^{\alpha_{1} ;
\psi }\displaystyle\left( \mathcal{W}_{v_{1},v_{2}}^{p,v}f\left( x,y\right)\right)
\end{array}%
\right\vert   \notag \\
&\leq &I_{0+;x}^{\alpha_{1} ;\psi }\left( \mathcal{W}^{p}\varphi \left( x,y\right)\right) ;
\end{eqnarray*}
\begin{eqnarray*}
&&\left\vert v_{2}\left( x,y\right) -\Psi^{\gamma }\left( 0,y\right) v_{2}\left( x,0\right) -I_{0+;y}^{\alpha_{2} ;\psi }f\left( x,y,v\left(x,y\right) ,v_{1}
\left( x,y\right) ,v_{2}\left( x,y\right) \right)\right\vert   \notag \\
&\leq &I_{0+;y}^{\alpha_{2} ;\psi }\left( \mathcal{W}^{p}\varphi \left( x,y\right)\right) ;
\end{eqnarray*}
$ x\in \left[ 0,a\right) ,$ $y\in \left[ 0,b\right) $, $\dfrac{2}{3}<\alpha \leq 1$, $0\leq \gamma \leq 1$, $v_{1}=
\dfrac{\partial _{\beta;\psi }^{\alpha }v}{\partial _{\beta ,;\psi }y^{\alpha }}$ and $v_{2}=
\dfrac{\partial _{\beta ;\psi }^{2\alpha }v}{\partial _{\beta ,;\psi }x^{2\alpha }}$.
\end{remark}

\section{Main results}

In this section we present the main results obtained in this paper, the Ulam-Hyers and generalized Ulam-Hyers-Rassias stabilities for the solution of the 
fractional partial differential equation of the pseudoparabolic type Eq.(\ref{eq3}), as well as the uniqueness of solutions. This section will be divided into 
two sub-sections, for better development and understanding of results.

\subsection{Ulam-Hyers Stability}

Now, we consider Ulam-Hyers stability of the \rm{{Eq.(\ref{eq3})}}. For this problem we have the following result. 
\begin{theorem} We suppose that 
\begin{enumerate}
\item $a<\infty $, $b<\infty $ 
\item $f\in C\left( \left[ 0,a\right) \times \left[ 0,b\right)\times \mathbb{B}^{3}, \mathbb{B}\right) $ 
\item $\exists L_{f}>0$ such that 
\begin{equation}
\left\vert f\left( x,y,u_{1},u_{2},u_{3}\right) -f\left( x,y,v_{1},v_{2},v_{3}\right) \right\vert \leq L_{f}\underset{i\in \{1,2,3 \} }{\max } 
\left\vert u_{i}-v_{i}\right\vert ;
\end{equation}
for all $x\in \left[ 0,a\right]$, $y\in \left[ 0,b\right] $ and $u_{1},u_{2},u_{3},v_{1},v_{2},v_{3}\in \mathbb{B}$.
\end{enumerate}

Then, we have
$\left( a\right) $ For $h\in C^{2}\left( \left[ 0,a\right], \mathbb{B}\right)$, $g_{1},g_{2}\in C^{1}\left( \left[ 0,b\right], \mathbb{B}\right) $ 
the {\rm{{Eq.(\ref{eq3})}}} has a unique solution with
\begin{equation}\label{k12}
\left\{ 
\begin{array}{ccc}
I_{\theta }^{1-\gamma ;\psi }u\left( x,0\right) & = & h\left( x\right), \text{ }x\in \left[ 0,a\right] \\ 
I_{\theta }^{1-\gamma ;\psi }u\left( 0,y\right) & = & g_{1}\left( y\right), \text{ }y\in \left[ 0,b\right] \\ 
u_{x}\left( 0,y\right) & = & g_{2}\left( y\right) ,\text{ }y\in \left[ 0,b\right]
\end{array}
\right.
\end{equation}
$\left( b\right) $ The {\rm{{Eq.(\ref{eq3})}}} is Ulam-Hyers stable.
\end{theorem}

\begin{proof}
(a) If $u\left( x,y\right) $ is a solution to the problem \rm{{Eq.(\ref{eq3})}} and \rm{{Eq.(\ref{k12})}}, then 
$\left( u,\dfrac{\partial _{\beta ;\psi }^{\alpha }u}{\partial _{\beta ;\psi }y^{\alpha }},\dfrac{\partial _{\beta ;\psi }^{2\alpha }u}{\partial
_{\beta ;\psi }x^{2\alpha }}\right) $ is a solution to the system:
\begin{equation}\label{k13}
\left\{ 
\begin{array}{cll}
u\left( x,y\right)  & = & \Psi ^{\gamma }\left( 0,y\right) h\left( x\right)
+\Psi ^{\gamma }\left( x,0\right) g_{1}\left( y\right) +\Psi ^{\gamma
}\left( x,0\right) xg_{2}\left( y\right) +I_{\theta }^{\alpha ;\psi }\left( 
\mathcal{W}_{u_{1},u_{2}}^{p,u}f\left( x,y\right) \right)  \\ 
u_{1}\left( x,y\right)  & = & \Psi ^{\gamma }\left( x,0\right) g_{1y^{\alpha
}}\left( y\right) +\Psi ^{\gamma }\left( x,0\right) xg_{2y^{\alpha }}\left(
y\right) +I_{0+;x}^{\alpha _{1};\psi }\left( \mathcal{W} _{u_{1},u_{2}}^{p,u}f\left( x,y\right) \right)  \\ 
u_{2}\left( x,y\right)  & = & \Psi ^{\gamma }\left( 0,y\right) h_{x^{2\alpha
}}\left( x\right) +I_{0+;y}^{\alpha _{2};\psi }f\left( x,y,u\left(
x,y\right) ,u_{1}\left( x,y\right) ,u_{2}\left( x,y\right) \right) 
\end{array}%
\right. 
\end{equation}
where $u_{1}\left( x,y\right) =\dfrac{\partial _{\beta ;\psi }^{\alpha }u}{\partial _{\beta ;\psi }y^{\alpha }}\left( x,y\right)$, $u_{2}
\left(x,y\right) =\dfrac{\partial _{\beta ;\psi }^{2\alpha }u}{\partial _{\beta ;\psi }x^{2\alpha }}\left( x,y\right)$, $g_{2y^{\alpha }}
\left(y\right)=\dfrac{\partial _{\beta ;\psi }^{2\alpha }u}{\partial _{\beta ;\psi}x^{2\alpha }}\left(y\right)$, $g_{1y^{\alpha }}
\left( y\right) =\dfrac{\partial _{\beta ;\psi }^{\alpha }u}{\partial _{\beta ;\psi }y^{\alpha }}\left(y\right) $ and $h_{x^{2\alpha }}
\left( x\right) =\dfrac{\partial_{\beta ;\psi }^{2\alpha }u}{\partial _{\beta ;\psi }x^{2\alpha }}\left(x\right)$.

We denote the right hand sides of the equation in \rm{{Eq.(\ref{k13})}} by $A_{1}$, $A_{2}, $ 
$A_{3}$, respectively. The system \rm{{Eq.(\ref{k13})}}, then becomes
\begin{equation*}
\left\{ 
\begin{array}{ccc}
u\left( x,y\right) & = & A_{1}\left( u,u_{1},u_{2}\right) \left( x,y\right)\\ 
u_{1}\left( x,y\right) & = & A_{2}\left( u,u_{1},u_{2}\right) \left(x,y\right) \\ 
u_{2}\left( x,y\right) & = & A_{3}\left( u,u_{1},u_{2}\right) \left(x,y\right)
\end{array}
\right. ,\text{ }u,u_{1},u_{2}\in C\left( \left[ 0,a\right) \times \left[0,b\right) \right).
\end{equation*}

Let $X:=C\left( \left[ 0,a\right) \times \left[ 0,b\right) \right) \times C\left( \left[ 0,a\right) \times \left[ 0,b\right) \right) \times C\left( 
\left[ 0,a\right) \times \left[ 0,b\right) \right) $ and for any $\delta >0$.

Consider the Bielecki norm on $X$:
\begin{equation*}
\left\Vert \left( u,u_{1},u_{2}\right) \right\Vert _{B}:=\max \left\{M_{1},M_{2},M_{3}\right\}
\end{equation*}
with 
\begin{equation*}
\left\{ 
\begin{array}{ccc}
M_{1} & := & \underset{\left( x,y\right) \in \left[ 0,a\right) \times \left[0,b\right) }{\max }\left\vert u\left( x,y\right) \right\vert \exp 
\left(-\delta \left( x+y\right) \right)  \\ 
M_{2} & := & \underset{\left( x,y\right) \in \left[ 0,a\right) \times \left[0,b\right) }{\max }\left\vert u_{1}\left( x,y\right) \right\vert \exp 
\left(-\delta \left( x+y\right) \right)  \\ 
M_{3} & := & \underset{\left( x,y\right) \in \left[ 0,a\right) \times \left[0,b\right) }{\max }\left\vert u_{2}\left( x,y\right) \right\vert \exp 
\left(-\delta \left( x+y\right) \right) 
\end{array}
\right. 
\end{equation*}

Then $\left( X,\left\Vert \cdot \right\Vert _{B}\right) $ is an ordered $L$-space. We will define the operator $A:X\rightarrow X$ by
\begin{equation*}
\left( u,u_{1},u_{2}\right) \rightarrow \left( A_{1}\left(u,u_{1},u_{2}\right) ,\text{ }A_{2}\left( u,u_{1},u_{2}\right), \text{ }A_{3}\left( u,u_{1},u_{2}\right) \right).
\end{equation*}

Using the hypotheses (1)-(3), we have
\begin{equation}\label{ze4}
\left\Vert A\left( \overline{u},\overline{u_{1}},\overline{u_{2}},\right)-A\left( \overline{\overline{u}},\overline{\overline{u_{1}}},\overline{\overline{u_{2}}},
\right) \right\Vert _{B}\leq \frac{L_{f}}{\delta }\left\Vert \left( \overline{u},\overline{u_{1}},\overline{u_{2}},\right)-\left( \overline{\overline{u}},
\overline{\overline{u_{1}}},\overline{\overline{u_{2}}},\right) \right\Vert _{B}.
\end{equation}

Taking $\delta >0$ such that $\dfrac{L_{f}}{\delta }<1$ in relation \rm{Eq.(\ref{ze4})}, the operator $A$ is a contraction and by the contraction principle 
the conclusion follows.

(b) Let $v$ be a solution to the inequality \rm{Eq.(\ref{eq5})}. Let $v$ be the unique solution of the \rm{{Eq.(\ref{eq3})}} satisfying the conditions:
\begin{equation}
\left\{ 
\begin{array}{ccc}
I_{\theta }^{1-\gamma ;\psi }u\left( x,0\right) & = & v\left( x,0\right), \text{ }x\in \left[ 0,a\right] \\ 
I_{\theta }^{1-\gamma ;\psi }u\left( x,0\right) & = & v\left( 0,y\right), \text{ }y\in \left[ 0,b\right] \\ 
u_{x}\left( 0,y\right) & = & v_{x}\left( 0,y\right), \text{ }y\in \left[ 0,b\right]
\end{array}
\right.
\end{equation}
with $\gamma=\alpha+\beta(1-\alpha)$.

From Remark \ref{re3}, the condition $\left( 3\right) $ and Lemma \ref{l1} (Gronwall Lemma), we have that
\begin{eqnarray*}
&&\left\vert v\left( x,y\right) -u\left( x,y\right) \right\vert   \notag \\
&=&\left\vert 
\begin{array}{c}
v\left( x,y\right) -\Psi^{\gamma }\left( 0,y\right)v\left( x,0\right) -\Psi^{\gamma }\left( x,0\right) v\left( 0,y\right) - \Psi^{\gamma }
\left( x,0\right)xv_{x}\left( 0,y\right)\\ -I_{\theta }^{\alpha;\psi }\displaystyle\left( \mathcal{W}_{u_{1},u_{2}}^{p,u}f\left( x,y\right)\right) 
\end{array}
\right\vert  \\
&\leq &\left\vert 
\begin{array}{c}
v\left( x,y\right) -\Psi^{\gamma }\left( 0,y\right) v\left( x,0\right) -\Psi^{\gamma }\left( x,0\right) v\left( 0,y\right) -\Psi^{\gamma }
\left( x,0\right) xv_{x}\left( 0,y\right)\\ -I_{\theta }^{\alpha;\psi }\displaystyle\left( \mathcal{W}_{v_{1},v_{2}}^{p,v}f\left( x,y\right)\right) 
\end{array}
\right\vert   \notag \\
&&+I_{\theta }^{\alpha ;\psi }\left(\widetilde{\mathcal{W}}_{v,u}^{p}f\left( x,y\right)\right)   \notag \\
&\leq &\varepsilon x\dfrac{\left( \psi \left( x\right) -\psi \left( 0\right) \right) ^{\alpha _{1}}\left( \psi \left( y\right) -\psi \left( 0\right) \right) ^{\alpha _{2}}}{\Gamma \left( \alpha _{1}+1\right) \Gamma \left( \alpha _{2}+1\right) }+L_{f}I_{\theta }^{\alpha ;\psi }\left( p\underset{i\in \{1,2,3 \}}{\max }\left\vert u_{i}\left( x,y\right) -v_{i}\left( x,y\right) \right\vert \right)   \notag \\
&\leq &\varepsilon a\dfrac{\left( \psi \left( b\right) -\psi \left( 0\right) \right) ^{\alpha _{2}}\left( \psi \left( a\right) -\psi \left( 0\right) \right) ^{\alpha _{1}}}{\Gamma \left( \alpha _{1}+1\right) \Gamma \left(
\alpha _{2}+1\right) }\mathbb{E}_{\alpha }\left[ L_{f}a\Gamma \left( \alpha _{1}\right) \Gamma \left( \alpha _{2}\right) \left( \psi \left( b\right) -\psi \left( 0\right) \right) ^{\alpha _{2}}\left( \psi \left( a\right)
-\psi \left( 0\right) \right) ^{\alpha _{1}}\right]   \notag \\ &=&C_{f}^{1}\varepsilon 
\end{eqnarray*}
where
\begin{equation*}
C_{f}^{1}:=a\frac{\left( \psi \left( b\right) -\psi \left( 0\right) \right) ^{\alpha _{2}}\left( \psi \left( a\right) -\psi \left( 0\right) \right) ^{\alpha _{1}}}{\Gamma \left( \alpha _{1}+1\right) \Gamma \left( \alpha
_{2}+1\right) }\mathbb{E}_{\alpha }\left[ L_{f}a\Gamma \left( \alpha _{1}\right) \Gamma \left( \alpha _{2}\right) \left( \psi \left( b\right) -\psi \left( 0\right) \right) ^{\alpha _{2}}\left( \psi \left( a\right) -\psi \left( 
0\right) \right) ^{\alpha _{1}}\right]
\end{equation*}
and $\mathbb{E}_{\alpha}(\cdot)$ is the one-parameter Mittag-Leffler function.

By performing the same process as above, we obtain the following inequalities
\begin{eqnarray*}
\left\vert v_{1}\left( x,y\right) -u_{1}\left( x,y\right) \right\vert  &\leq &\varepsilon a\frac{\left( \psi \left( a\right) -\psi \left( 0\right) \right) ^{\alpha _{1}}}{\Gamma \left( \alpha _{1}+1\right) }\mathbb{E}_{\alpha }\left[ L_{f}a\left( \psi \left( a\right) -\psi \left( 0\right) \right) ^{\alpha _{1}}\Gamma \left( \alpha _{1}\right) \right]   \notag \\ &=&C_{f}^{2}\varepsilon 
\end{eqnarray*}
where $C_{f}^{2}:=a\dfrac{\left( \psi \left( a\right) -\psi \left( 0\right) \right) ^{\alpha _{1}}}{\Gamma \left( \alpha _{1}+1\right) }\mathbb{E}_{\alpha }\left[ L_{f}a\left( \psi \left( a\right) -\psi \left( 0\right) \right) ^{\alpha _{1}}\Gamma \left( \alpha _{1}\right) \right] $
and
\begin{eqnarray*}
\left\vert v_{2}\left( x,y\right) -u_{2}\left( x,y\right) \right\vert  &\leq &\varepsilon \frac{\left( \psi \left( b\right) -\psi \left( 0\right) \right) ^{\alpha 2}}{\Gamma \left( \alpha _{2}+1\right) }\mathbb{E}_{\alpha }\left[ L_{f}\left( \psi \left( b\right) -\psi \left( 0\right) \right) ^{\alpha _{2}}\right]   \notag \\
&=&C_{f}^{3}\varepsilon 
\end{eqnarray*}
where $C_{f}^{3}:=\dfrac{\left( \psi \left( b\right) -\psi \left( 0\right) \right) ^{\alpha _{2}}}{\Gamma \left( \alpha _{2}+1\right) }\mathbb{E}_{\alpha }\left[ L_{f}\left( \psi \left( b\right) -\psi \left( 0\right) \right) ^{\alpha _{2}} \right] $.

So, the \rm{Eq.(\ref{eq3})} is Ulam-Hyers stable.
\end{proof}


\subsection{Generalized Ulam-Hyers-Rassias Stability}

\begin{theorem} We suppose that 
\begin{enumerate}
\item $f\in C\left( \left[ 0,\infty \right) \times \left[ 0,\infty \right) \times \mathbb{B}^{3}, \mathbb{B}\right)$; \item 
There exists $L_{f}\in C^{1}\left( \left[ 0,\infty \right) \times \left[ 0,\infty \right), \mathbb{R}_{+}\right) ,$ such that
\begin{equation*}
\left\vert f\left( x,y,u_{1},u_{2},u_{3}\right) -f\left( x,y,v_{1},v_{2},v_{3}\right) \right\vert \leq L_{f}\left( x,y\right) \underset{i\in \{1,2,3 \} }{\max }\left\vert u_{i}-v_{i}\right\vert 
\end{equation*}
for all $x,y\in \left[ 0,\infty \right) $, $u_{1},u_{2},u_{3},v_{1},v_{2},v_{3}\in \mathbb{B}$;

\item There exists $\lambda _{\varphi }^{1},\lambda _{\varphi }^{2},\lambda _{\varphi }^{3}>0$ such that
\begin{equation*}
\left\{ 
\begin{array}{ccc}
I_{\theta }^{\alpha ;\psi }\displaystyle\left( \mathcal{W}^{p}\varphi \left( x,y\right)\right)  & \leq  & \lambda _{\varphi }^{1}\varphi \left( x,y\right)  \\ 
I_{0+;x}^{\alpha _{1};\psi }\displaystyle\left( \mathcal{W}^{p}\varphi \left( x,y\right)\right)  & \leq  & \lambda _{\varphi }^{2}\varphi \left( x,y\right)  \\ 
I_{0+;y}^{\alpha _{2};\psi }\displaystyle\left( \mathcal{W}^{p}\varphi \left( x,y\right)\right)  & \leq  & \lambda _{\varphi }^{3}\varphi \left( x,y\right) 
\end{array}%
\right. 
\end{equation*}
$x,y\in \left[ 0,\infty \right)$.

\item $\varphi:\mathbb{R}_{+}\times \mathbb{R} _{+}\rightarrow \mathbb{R}_{+}$ is increasing.
\end{enumerate}

Then the {\rm{Eq.(\ref{eq3})}} $\left( a=\infty \text{, }b=\infty \right) $ is generalized Ulam-Hyers-Rassias stable.
\end{theorem}

\begin{proof}
Let $v$ be a solution of the inequality \rm{Eq.(\ref{eq5})} and $u\left( x,y\right) $ the unique solution of the problem
\begin{equation*}
\left\{ 
\begin{array}{cll}
\dfrac{\partial _{\beta ;\psi }^{3\alpha }u}{\partial _{\beta ;\psi }x^{2\alpha }\partial _{\beta ;\psi }y^{\alpha }}\left( x,y\right)  & = & 
f\left( x,y,u\left( x,y\right) ,\dfrac{\partial _{\beta ;\psi }^{\alpha }u}{\partial _{\beta ;\psi }y^{\alpha }}\left( x,y\right) ,
\dfrac{\partial_{\beta ;\psi }^{\alpha }u}{\partial _{\beta ;\psi }x^{2\alpha }}\left( x,y\right) \right) \\ 
I_{\theta }^{1-\gamma ;\psi }u\left( x,0\right)  & = & v\left( x,0\right) \\ 
I_{\theta }^{1-\gamma ;\psi }u\left( 0,y\right)  & = & v\left( 0,y\right) \\ 
v_{x}\left( 0,y\right)  & = & v_{x}\left( 0,y\right)
\end{array}
\right.
\end{equation*}
$x,y\in \left[ 0,\infty \right)$ and with $\gamma=\alpha+\beta(1-\alpha)$, then $\left( u,u_{1},u_{2}\right) $ is a solution of the system:
\begin{equation}
\left\{ 
\begin{array}{cll}
u\left( x,y\right)  & = & \Psi ^{\gamma }\left( 0,y\right) v\left(
x,0\right) +\Psi ^{\gamma }\left( x,0\right) v\left( 0,y\right) +\Psi
^{\gamma }\left( x,0\right) xv_{x}\left( 0,y\right) +I_{\theta }^{\alpha
;\psi }\left( \mathcal{W}_{u_{1},u_{2}}^{p,u}f\left( x,y\right) \right)  \\ 
u_{1}\left( x,y\right)  & = & \Psi ^{\gamma }\left( x,0\right) v_{1}\left(
0,y\right) +\Psi ^{\gamma }\left( x,0\right) xv_{1x}\left( 0,y\right)
+I_{0+;x}^{\alpha ;\psi }\left( \mathcal{W}_{u_{1},u_{2}}^{p,u}f\left(
x,y\right) \right)  \\ 
u_{2}\left( x,y\right)  & = & \Psi ^{\gamma }\left( 0,y\right) v_{2}\left(
x,0\right) +I_{\theta }^{\alpha ;\psi }\left( f\left( x,y,u\left( x,y\right)
,u_{1}\left( x,y\right) ,u_{2}\left( x,y\right) \right) \right) 
\end{array}%
\right. 
\end{equation}

On the other hand, by Remark \ref{re4} and using the hypotheses $\left( 3\right)$, we get
\begin{eqnarray}\label{x1}
&&\left\vert 
\begin{array}{c}
v\left( x,y\right) - \Psi^{\gamma }\left( 0,y\right)v\left( x,0\right) - \Psi^{\gamma }\left( x,0\right) v\left( 0,y\right) -\Psi^{\gamma }
\left( x,0\right) xv_{x}\left( 0,y\right) \\-I_{\theta }^{\alpha ;\psi }\displaystyle\left( \mathcal{W}_{u_{1},u_{2}}^{p,u}f\left( x,y\right)\right) 
\end{array}
\right\vert   \notag \\
&\leq &I_{\theta }^{\alpha ;\psi }\left( \mathcal{W}^{p}\varphi \left( x,y\right)\right) \leq \lambda _{\varphi }^{1}\varphi \left( x,y\right) ,%
\text{ } x,y\in \left[ 0,\infty \right);
\end{eqnarray}
\begin{eqnarray}\label{x2}
&&\left\vert 
\begin{array}{c}
v_{1}\left( x,y\right) -\Psi^{\gamma }\left( x,0\right) v_{1}\left( 0,y\right) -\Psi^{\gamma }\left( x,0\right)xv_{1x}\left( 0,y\right) -
I_{0+;x}^{\alpha _{1};\psi }\displaystyle\left( \mathcal{W}_{u_{1},u_{2}}^{p,u}f\left( x,y\right)\right) 
\end{array}%
\right\vert   \notag \\
&\leq &I_{0+;x}^{\alpha _{1};\psi }\left( \mathcal{W}^{p}\varphi \left( x,y\right)\right) \leq \lambda _{\varphi }^{2}\varphi \left( x,y\right), 
\text{ }x,y\in \left[ 0,\infty \right); 
\end{eqnarray}%
\begin{eqnarray}\label{x3}
&&\left\vert v_{2}\left( x,y\right) - \Psi^{\gamma }\left( 0,y\right) v_{2}\left( x,0\right) -I_{0+;y}^{\alpha _{2};\psi }\left( f\left(x,y,u
\left( x,y\right) ,u_{1}\left( x,y\right) ,u_{2}\left( x,y\right) \right) \right) \right\vert   \notag \\
&\leq &\lambda _{\varphi }^{3}\varphi \left( x,y\right) ,\text{ }x,y\in \left[ 0,\infty \right).
\end{eqnarray}

In this sense, from \rm{Eq.(\ref{x1})}, \rm{Eq.(\ref{x2})}, \rm{Eq.(\ref{x3})} and using the Gronwall Lemma \ref{l1}, we obtain
\begin{eqnarray}\label{x5}
&&\left\vert v\left( x,y\right) -u\left( x,y\right) \right\vert   \notag \\
&\leq &\left\vert 
\begin{array}{c}
v\left( x,y\right) -\Psi^{\gamma }\left( 0,y\right) v\left( x,0\right) - \Psi^{\gamma }\left( x,0\right) v\left( 0,y\right) -\Psi^{\gamma }\left( x,0\right) xv_{x}\left( 0,y\right)\\ -I_{\theta
}^{\alpha ;\psi }\displaystyle\left( \mathcal{W}_{v_{1},v_{2}}^{p,v}f\left( x,y\right)\right) 
\end{array}%
\right\vert   \notag \\
&&+I_{\theta }^{\alpha ;\psi }\left( \widetilde{\mathcal{W}}_{v,u}^{p}f\left( x,y\right)\right)   \notag \\
&\leq &\lambda _{\varphi }^{1}\varphi \left( x,y\right) +I_{\theta }^{\alpha
;\psi }\left( \int_{0}^{p}L_{f}\left( x,y\right) \underset{i\in \{1,2,3 \} }{\max }\left\vert v_{i}\left( x,y\right) -u_{i}\left( x,y\right)
\right\vert dp\right)   \notag \\
&\leq &\lambda _{\varphi }^{1}\varphi \left( x,y\right) \mathbb{E}_{\alpha }%
\left[ \int_{0}^{p}L_{f}\left( x,y\right) \Gamma \left( \alpha _{1}\right)
\Gamma \left( \alpha _{2}\right) \left( \psi \left( b\right) -\psi \left(
0\right) \right) ^{\alpha _{2}}\left( \psi \left( a\right) -\psi \left(
0\right) \right) ^{\alpha _{1}}dp\right]   \notag \\
&\leq &\lambda _{\varphi }^{1}\varphi \left( x,y\right) \mathbb{E}_{\alpha }%
\left[ \int_{0}^{p}L_{f}\left( x,y\right) \Gamma \left( \alpha _{1}\right)
\Gamma \left( \alpha _{2}\right) \left( \psi \left( \infty \right) -\psi
\left( 0\right) \right) ^{\alpha _{2}}\left( \psi \left( \infty \right)
-\psi \left( 0\right) \right) ^{\alpha _{1}}dp\right]   \notag \\
&=&C_{f;\varphi }^{1}\varphi \left( x,y\right) 
\end{eqnarray}
where $C_{f;\varphi }^{1}:=\lambda _{\varphi }^{1}\mathbb{E}_{\alpha }\displaystyle\left[ \int_{0}^{p}L_{f}\left( x,y\right) \Gamma 
\left( \alpha _{1}\right) \Gamma \left( \alpha _{2}\right) \left( \psi \left( \infty \right) -\psi \left( 0\right) 
\right) ^{\alpha _{2}}\left( \psi \left( \infty \right) -\psi \left( 0\right) \right) ^{\alpha _{1}}dp\right] ,$ 
$\dfrac{2}{3}<\alpha \leq 1 $ and $\psi \left( \infty \right) <\infty ,$ $x,y\in \left[ 0,\infty \right)$ and 
$\mathbb{E}_{\alpha}(\cdot)$ is the one-parameter Mittag-Leffler function.

Performing the same steps as above to obtain the inequality \rm{Eq.(\ref{x5})}, we can write
\begin{equation*}
\left\vert v_{1}\left( x,y\right) -u_{1}\left( x,y\right) \right\vert \leq C_{f,\varphi }^{2}\varphi \left( x,y\right) 
\end{equation*}
where $C_{f,\varphi }^{2}:=\lambda _{\varphi }^{2}\mathbb{E}_{\alpha }\displaystyle\left[ \int_{0}^{p}L_{f}\left( x,y\right) \Gamma 
\left( \alpha _{1}\right) \left( \psi \left( \infty \right) -\psi \left( 0\right) \right) ^{\alpha _{1}}dp \right] ,$ $\dfrac{2}{3}<\alpha 
\leq 1$ and $\psi \left( \infty \right) <\infty$, $ x,y\in \left[ 0,\infty \right)$.

Also, we get
\begin{equation*}
\left\vert v_{2}\left( x,y\right) -u_{2}\left( x,y\right) \right\vert \leq C_{f,\varphi }^{3}\varphi \left( x,y\right) 
\end{equation*}
where $C_{f,\varphi }^{3}:=\lambda _{\varphi }^{3}\mathbb{E}_{\alpha }\displaystyle\left[ \int_{0}^{p}L_{f}\left( x,y\right) \Gamma 
\left( \alpha _{2}\right) \left( \psi \left( \infty \right) -\psi \left( 0\right) \right) ^{\alpha _{2}}dp \right]$, $\dfrac{2}{3}<\alpha 
\leq 1$ and $\psi \left( \infty \right) <\infty$, $x,y\in \left[ 0,\infty \right)$.

Then, the solution of the \rm{Eq.(\ref{eq3})} is generalized Ulam-Hyers-Rassias stable.

\end{proof}

\section{Concluding remarks}

In this paper, we have presented and proposed new stability results of the Ulam-Hyers type of solutions of fractional partial differential equations, contributing to the diffusion of results in the fractional calculus theme, particularly, in the fractional analysis, significantly enriching this field of study.

In this sense, we have been successful in presenting stability results of Ulam-Hyers and generalized Ulam-Hyers-Rassias solution of a fractional order pseudoparabolic partial differential equation, using the $\psi$-Hilfer fractional partial derivative of $N$ variables and the Gronwall inequality. It is also possible to note that there are authors who use in their works, the Banach fixed-point theorem to discuss stability \cite{esto1,esto4}.

As mentioned in the introduction, the study of the stability of solutions of fractional differential equations is in a growing, and numerous researchers have contributed to such advancement. However, it is still possible to obtain other new Ulam-Hyers stability results from other types of partial differential equations: hyperbolic, parabolic and elliptical. On the other hand, a theme that gains special attention is the study of stability of the solution of the linear heat equation, by means of integral transforms, that is, the Fourier transform and the 
Laplace transform. Studies in this direction must be presented and published in the near future. 
\bibliography{ref}
\bibliographystyle{plain}

\end{document}